\documentclass{daj}
\usepackage{amssymb,amsmath, amsthm, amscd,ifthen}

\newcommand{\ff}{{\mathcal F}}
\newcommand{\aaa}{{\mathcal A}}

\newcommand{\G}{{\mathcal G}}

\newcommand{\bb}{{\mathcal B}}

\newtheorem{thm}{Theorem}

\newtheorem{lem}[thm]{Lemma}
\newtheorem{cla}[thm]{Claim}
\newcommand{\eps}{{\varepsilon}}

\newtheorem{cor}[thm]{Corollary}
\date{}

\newtheorem{prop}[thm]{Proposition}

\newcommand{\E}{\mathrm E}

\dajAUTHORdetails{%
  title = {Simple Juntas for Shifted Families}, 
  author = {Peter Frankl and Andrey Kupavskii},
  plaintextauthor = {Peter Frankl, Andrey Kupavskii},
  plaintexttitle = {Simple Juntas for Shifted Families}, 
  runningtitle = {Simple Juntas for Shifted Families},
  runningauthor = {Peter Frankl, Andrey Kupavskii},
  copyrightauthor = {P. Frankl, A. Kupavskii},
  keywords = {extremal set theory, shifting, juntas},
}   

\dajEDITORdetails{%
   year={2020},
   number={14},
   received={26 January 2019},   
   published={3 September 2020},  
   doi={10.19086/da.14507},       
}   

\begin{document}

\begin{frontmatter}[classification=text]

\title{Simple Juntas for Shifted Families} 

\author[peter]{Peter Frankl}
\author[andrey]{Andrey Kupavskii}

\begin{abstract}
We say that a family $\mathcal F$ of $k$-element sets is a {\it $j$-junta} if there is a set $J$ of size $j$ such that, for any $F$, its presence in $\ff$ depends on its intersection with $J$ only. Approximating arbitrary families by $j$-juntas with small $j$ is a recent powerful technique in extremal set theory.

The weak point of all known  junta approximation results is that they work in the range $n>Ck$, where $C$ is an extremely fast growing function of the input parameters, such as the quality of approximation or the number of families we simultaneously approximate.

We say that a family $\ff$ is {\it shifted} if for any $F=\{x_1,\ldots, x_k\}\in \ff$ and any $G =\{y_1,\ldots, y_k\}$ such that $y_i\le x_i$, we have $G\in \ff$. For many extremal set theory problems, including the Erd\H os Matching Conjecture, or the Complete $t$-Intersection Theorem, it is sufficient to deal with shifted families only.

  In this paper, we present very general approximation by juntas results for shifted families with explicit (and essentially linear) dependency on the input parameters. The results are best possible up to some constant factors. Moreover, they give meaningful statements for almost all range of values of $n$. The proofs are shorter than the proofs of the previous approximation by juntas results and are completely self-contained.

  As an application of our junta approximation, we give a nearly-linear bound for the multi-family version of the Erd\H os Matching Conjecture. More precisely, we prove the following result. Let $n\ge 12sk\log(e^2s)$ and suppose that the families $\ff_1,\ldots, \ff_s\subset {[n]\choose k}$ do not contain $F_1\in\ff_1,\ldots, F_s\in \ff_s$ such that $F_i$'s are pairwise disjoint. Then $\min_{i}|\ff_i|\le {n\choose k}-{n-s+1\choose k}.$
\end{abstract}
\end{frontmatter}


\section{Introduction}

Let $[n]:=\{1,\ldots, n\}$. We use the standard notation $2^{[n]}$ for the power set of $[n]$ and ${[n]\choose k}$ for the family of all $k$-element subsets of $[n]$. We say that a family $\mathcal J\subset 2^{[n]}$ is a {\it $j$-junta}, if there exists a set $J\in {[n]\choose j}$ and a family $\mathcal J^*\subset 2^{J}$ such that $\mathcal J = \{F\subset [n]: F\cap J\in\mathcal J^*\}$. We call $J$ as above the {\it center} of the junta and $\mathcal J^*$ the {\it defining family}.

We say that $\mathcal F\subset 2^{[n]}$ is {\it $t$-intersecting} if $|F_1\cap F_2|\ge t$ for any two $F_1,F_2\in \ff$. We say ``intersecting'' instead of ``$1$-intersecting'' for shorthand. Similarly, if for any $A\in\aaa\subset 2^{[n]}$ and $B\in \bb\subset 2^{[n]}$ we have $|A\cap B|\ge t$, then we say that $\aaa$ and $\bb$ are cross $t$-intersecting. We say ``cross-intersecting'' instead of ``cross $1$-intersecting'' for shorthand.
In a seminal paper \cite{DF}, Dinur and Friedgut proved the following theorem, which, informally speaking, states that the structure of large intersecting families is governed by a small number of coordinates.

\begin{thm}[\cite{DF}]\label{thmdf} There exist functions $j(r),c(r)$ such that for any integers $1<j(r)<k<n/2$, if $\aaa\subset {[n]\choose k}$ is an intersecting family then there exists an intersecting $j$-junta $\mathcal J$ with $j\le j(r)$ and
$$|\mathcal F\setminus \mathcal J|\le c(r)\cdot {n-r\choose k-r}.$$
\end{thm}
This result is, in fact, a corollary of the analogous statement concerning cross-intersecting families.
\begin{thm}[\cite{DF}]\label{thmdf2} There exist functions $j(r),c(r)\ge 1$ such that the following holds. Let $n,a,b\in \mathbb N$ with $1<j(r)<b\le a$ and $a+b<n$ and let $\aaa\subset {[n]\choose a},\ \bb\subset {[n]\choose b}$ be cross-intersecting.
Then there exist cross-intersecting $j$-juntas $\mathcal J$ and $\mathcal I$ with $j\le j(r)$, such that
$$|\mathcal A\setminus \mathcal J|\le c(r)\cdot {n-r\choose a-r}\ \ \ \text{and} \ \ \ |\mathcal B\setminus \mathcal J|\le c(r)\cdot {n-r\choose b-r}.$$
\end{thm}

Clearly, these results are meaningful only for $n>Ck$ and $n>Ca$, respectively, where $C=C(r)$ is sufficiently large (otherwise, say, in the first result, we may have $C(r){n-r\choose k-r}\ge {n\choose k}$, in which case the displayed inequality becomes trivial).

Recently, this result was greatly extended by Keller and Lifshitz \cite{KLchv} to the so-called `Turan problems for expansions'. This, in particular, includes the case of cross-dependent families. We say that $\ff_1,\ldots, \ff_s$ are {\it cross-dependent} if there is no choice $F_1\in\ff_1\ldots, F_s\in \ff_s$ such that $F_i$ are pairwise disjoint.
\begin{thm}[\cite{KLchv}]\label{thmkl} Let $s,r$ be some constants, $k<\frac n{2s}$, and consider cross-dependent families $\ff_1,\ldots, \ff_s\subset {[n]\choose k}$. Then there exist $C= C(s,r)$ and cross-dependent $C$-juntas $\mathcal J_1,\ldots, \mathcal J_s\subset {[n]\choose k}$, such that for each $i\in [s]$
$$|\mathcal F_i\setminus \mathcal J_i|\le C\cdot {n-r\choose k-r}.$$
\end{thm}

Again, the theorem only makes sense for $n>C'(s,r) k$, and the dependence of $C'$ on $s,r$ is not explicit (and at least exponential in $s,r$). In many cases, it is desirable to have a control of the behaviour of $C'(s,r)$. In particular, this is the case for the Erd\H os Matching Conjecture (cf. \cite{E}, \cite{F4},\cite{FK16}). We also note that, in some related settings, the exponential dependence cannot be improved. This is, e.g., the case of Friedgut's junta theorem \cite{Fri1}.

We say that the family $\ff$ is {\it shifted} if for any $F=\{x_1,\ldots, x_k\}\in \ff$ and any $G =\{y_1,\ldots, y_k\}$ such that $y_i\le x_i$, we have $G\in \ff$. One makes the family shifted by performing shifts (see definition in Section~\ref{sec32}). We refer to the survey of the first author \cite{Fra3}. Shifting is a very useful combinatorial operation, which preserves many properties of a family. In particular, it preserves the size of each set and the family. Moreover, it preserves the property of being (cross-) $t$-intersecting, cross-dependent etc. Thus, for many extremal problems (including the Erd\H os Matching Conjecture, Ahlswede--Khachatrian  theorem \cite{AK} and their common generalization \cite{FK20}), it is sufficient to restrict oneself to shifted families.



The purpose of this note is to show that one can obtain junta-type results for shifted families with essentially best possible dependencies on the parameters (see the remark after Theorem~\ref{thmshiftjunta1}) with purely combinatorial techniques.\footnote{All previous junta-type theorems rely on results from discrete Fourier analysis.} In the next section, we illustrate our ideas by giving a simple proof of a stronger version of Theorem~\ref{thmdf2} for shifted families. In Section~\ref{sec3}, we give a general junta approximation-type statement and deduce several of its combinatorial implications, including a  stronger version of Theorem~\ref{thmkl} for shifted families. In Section~\ref{sec4}, we present the proofs of the junta approximation results. In Section~\ref{sec5}, we apply these results to obtain the result on the multi-family version of the Erd\H os Matching Conjecture mentioned in the abstract.
\section{Cross $t$-intersecting families}
Let us first illustrate our methods in the setting of Theorem~\ref{thmdf2}.

\begin{thm}\label{thmshiftjuntas} Fix integers $a\ge b\ge r> t>0$ and shifted cross $t$-intersecting families $\aaa\subset {[n]\choose a}$ and $\bb\subset {[n]\choose b}$. Put $j:=2r-t-1$. If $n\ge 2a$  then there exist cross $t$-intersecting $j$-juntas $\mathcal J$  and $\mathcal I$ with center $[j]$ such that
\begin{equation}\label{firstjunta}|\mathcal A\setminus \mathcal J|\le 2^j\cdot {n-j\choose a-r}\ \ \ \ \text{and} \ \ \ \ |\mathcal B\setminus \mathcal I|\le 2^j\cdot {n-j\choose b-r}.\end{equation}
\end{thm}

For a family $\ff\subset 2^{[n]}$ and sets $X\subset S\subset [n]$ we use the following notation:
$$\ff(X,S):=\big\{F\setminus X\ :\ F\in\ff,\ F\cap S = X\big\}.$$
Note that $\ff(X,S)$ is regarded as a subfamily of $2^{[n]\setminus S}$.

In the proof of this theorem, we will use the following two lemmas. The first lemma was proved by the first author (cf., e.g., Proposition 9.3 in \cite{Fra3}).
\begin{lem}\label{lemshift} Assume that $\aaa,\bb$ are cross $t$-intersecting and shifted. Fix $X,Y, s$, such that $X,Y\subset [s]$ and $|X\cap Y|\le t-1$. Then $\aaa(X,[s])$ and $\bb(Y,[s])$ are cross $(t+s-|X|-|Y|)$-intersecting.
\end{lem}

The second lemma was essentially proven in \cite{F1}. We give its proof for completeness.
\begin{lem}\label{lemcross} Let $\aaa'\subset {[n']\choose a'}$ and $\bb'\subset {[n']\choose b'}$ be cross $t'$-intersecting. If $n'\ge 2\max \{a',b'\}$ then either $|\aaa'|\le {n'\choose a'-t'}$ or $|\bb'|\le {n'\choose b'-t'}$.
\end{lem}
For the sake of completeness, we shall give its proof.
\begin{proof}[Proof of Lemma~\ref{lemcross}]
The proof consists of two propositions. Let us say that a set $F\subset [n]$ has property $t'$ if there exists some $i\ge 0$ such that
$$|F\cap [t'+2i]|\ge t'+i.$$
\begin{prop}[\cite{F1}]
  The number of sets $F\in {[n]\choose k}$, $k\ge t'$, having property $t'$ is ${n\choose k-t'}.$
\end{prop}
\begin{proof}
Interpret each $k$-element set $F$ as a random walk on $\mathbb Z^2$ from point $(0,-t')$ to $(n-k,k-t')$, where at step $i$ we go one up or one right depending on whether the $i$'th element belongs or does not belong to $F$, respectively. Then the sets having property $t'$ are exactly the ones that correspond to random walks that hit the line $x=y$. By reflection principle, the number of such walks is equal to the number of shortest paths from $(-t',0)$ to $(n-k,k-t')$, which is exactly ${n\choose k-t'}$.
\end{proof}
\begin{prop} Suppose that the shifted families $\aaa',\bb'\subset 2^{[n]}$ are cross $t'$-intersecting. Then either all $A\in \aaa'$ have property $t'$ or all $B\in \bb'$ have property $t'+1$.
\end{prop}
\begin{proof} Suppose that $A\in \aaa'$ does not have property $t'$. Then, by shiftedness, a subset $(1,2,\ldots, t'-1, t'+1,t'+3,t'+5,\ldots)$ is in $\aaa'$. By the cross $t'$-intersecting property, no subset of $(1,2,\ldots, t',t'+2,t'+4,t'+6,\ldots)$ is in $\bb'$. By shiftedness, this implies that $\bb'$ has property $t'+1$.
\end{proof}
Together, these two propositions imply that either $|\aaa'|\le {n'\choose a'-t'}$ or $|\bb'|\le {n'\choose b'-t'-1}$.
\end{proof}

\begin{proof}[Proof of Theorem~\ref{thmshiftjuntas}]

Let us put
\begin{alignat*}{2}\mathcal J^*:=\ &\Big\{J\subset [j]\ :\  |\aaa(J,[j])|> {n-j\choose a-r}\Big\}, \ \ \ \ &&\mathcal J:=\ \Big\{A\in{n\choose a}\ :\ A\cap [j]\in \mathcal J^* \Big\};\\
\mathcal I^*:=\ &\Big\{I\subset [j]\ :\  |\bb(I,[j])|> {n-j\choose b-r}\Big\}, \ \ \ \ &&\mathcal I:=\ \Big\{B\in{n\choose b}\ :\ B\cap [j]\in \mathcal J^* \Big\}.
\end{alignat*}
Note that $\mathcal J^*$ and $\mathcal I^*$ may contain the empty set. We claim that $\mathcal J$ and $\mathcal I$ are the desired juntas. First, by the definition, it is clear that
$|\aaa\setminus \mathcal J|\le 2^{j} {n-j\choose a-r}$, and similarly for $\bb$ and $\mathcal I$.

Second, assume that $\mathcal J^*$ and $\mathcal I^*$ are not cross $t$-intersecting. This implies that there are two sets $X\in\mathcal J^*$ and $Y\in \mathcal I^*$ such that $|X\cap Y|\le t-1$. Denote $x:=|X|, y:=|Y|$. Consider, the families $\aaa':=\aaa(X,[j])$ and $\bb':=\bb(Y,[j])$. Note that the former family contains $(a-x)$-element sets and the latter contains $(b-y)$-element sets.
Due to Lemma~\ref{lemshift}, $\aaa'$ and $\bb'$ are cross $t'$-intersecting, where $t' = t+j-x-y=2r-1-x-y$. At the same time,
\begin{align}\label{eqa1}
 |\aaa'|>&\ {n-j\choose a-r},\\
\label{eqb1} |\bb'|>&\ {n-j\choose b-r}
\end{align}
 by the definition of $\mathcal J,\mathcal I$. In particular, $x,y\le r-1$.

Applying Lemma~\ref{lemcross} to our situation with $n' := n-j,$ $a': = a-x$, $b':= b-y$, $t':=2r-1-x-y$, we get that at least one of the following two inequalities is valid:
\begin{align}\label{eqa2}
  |\aaa'|\le&\ {n-j\choose (a-x)-t'}, \\
\label{eqb2}  |\bb'|\le&\ {n-j\choose (b-x)-t'}.
\end{align} We have $x,y\le r-1$ and $j\ge r$, and thus $(a-x)-t'=a-r-(r-1-y)\le a-r$. Thus, \eqref{eqa2} contradicts \eqref{eqa1}. But, similarly, \eqref{eqb2} contradicts \eqref{eqb1}. Therefore, we conclude that choosing such $X$ and $Y$ was impossible in the first place. That is, $\mathcal J^*$ and $\mathcal I^*$, and therefore $\mathcal J$ and $\mathcal I$, are cross $t$-intersecting.
\end{proof}

One may argue that the constants in \eqref{firstjunta} are still quite bad. Let us derive the following corollary, showing that, at the expense of slightly worse bounds on $n$ and $j$, one can get rid of the constants.

\begin{cor}
  \label{corshiftjuntas} Fix integers $a\ge b\ge r> t>0$ and shifted cross $t$-intersecting families $\aaa\subset {[n]\choose a}$ and $\bb\subset {[n]\choose b}$. For any $\epsilon>0$, put $j=2cr-t-1$, where $c:=1+\Big(\frac {2+\epsilon}{2\epsilon}\Big)^2\log_e 4$. If $n\ge (2+\epsilon)a$  then there exists cross $t$-intersecting $j$-juntas $\mathcal J$  and $\mathcal I$ with center $[j]$ such that
\begin{equation}\label{secondjunta}|\mathcal A\setminus \mathcal J|\le {n-r\choose a-r}\ \ \ \ \text{and} \ \ \ \ |\mathcal B\setminus \mathcal J|\le {n-r\choose b-r}.\end{equation}
\end{cor}
\begin{proof} Apply Theorem~\ref{thmshiftjuntas} $\aaa,\bb$ with $cr$ playing the role of $r$. Then we get cross $t$-intersecting juntas $\mathcal I,\mathcal J$ with center $[j]$, satisfying \eqref{firstjunta}. To prove the corollary, it is sufficient to show the second inequality in the following chain of inequalities:
$$2^{j}{n-2cr+t+1\choose k-cr}\le 2^{2cr}{n-(2c-1)r\choose k-cr}\le {n-r\choose k-r},$$
where $k\in \{a,b\}$. 
We have
\begin{align*}\frac{{n-(2c-1)r\choose k-cr}}{{n-r\choose k-r}}\le&\ \Big(\frac {k(n-k)}{n^2}\Big)^{(c-1)r}\le \Big(\frac {1+\epsilon}{(2+\epsilon)^2}\Big)^{(c-1)r}=4^{-(c-1)r}\Big(1-\Big(\frac {2\epsilon}{2+\epsilon}\Big)^2\Big)^{(c-1)r}\\
\le&\ 4^{-(c-1)r}e^{-(c-1)\big(\frac {2\epsilon}{2+\epsilon}\big)^2r}=4^{-cr},\end{align*}
where the second inequality in the first line holds due to the bound on $n$ and the fact that $x(1-x)$ is the biggest if $x$ is as close to $1/2$ as possible, and the last inequality holds due to the choice of $c$.
\end{proof}

\section{The general statement and its implications}\label{sec3}
In this section, we state a general junta approximation result and then deduce several corollaries for shifted families satisfying different properties. Let us present a geometric interpretation of the property that we are working with.

First, consider the case of one family $\ff$. Imagine that we are given an $(n+1)\times (n+1)$ grid (thought of as a subset of the plane) and we are walking on the grid starting from $(0,0)$ and  at each step we can either go from $(i,j)$ to $(i+1,j)$ or to $(i+1,j+1)$. The steps $\{i_1,\ldots, i_s\}$ of the walk at which the random walk goes ``diagonally'' give us a set in $2^{[n]}$ and, obviously, there is a one-to-one correspondence between the subsets in $2^{[n]}$ and such walks. The walks with exactly $k$ diagonal steps correspond to $k$-element subsets of $[n]$. Thus, any family $\ff$ corresponds to a collection of walks on this grid.

Next, consider the line $\lambda$ defined by $\alpha y=x+q$ for some $\alpha,q>0$ (here, by $x$ and $y$ we mean the coordinates on the plane). The property that we are interested in is ``any walk from $\ff$ hits $\lambda$ for some $x$'', where by hitting $\lambda$ we simply mean that $\alpha y\ge x+q$ for the corresponding $(x,y)$.\footnote{Of course, $(x,y)$ is an integer point.}

As was noted long time ago by the first author \cite{F1, F8}, there is a close relationship between different combinatorial properties of $\ff$ and the property that $\ff$ hits $\lambda$ for appropriately chosen $\lambda$. Here, we explore this relationship further. Actually, this relationship is the most transparent for the case of several families (i.e., in the ``cross setting'').

It is easy to generalize the  property of ``hitting the line'' to the case of several families $\ff_1,\ldots, \ff_s$.  Consider a hyperplane $\pi$ defined by $\sum_{i=1}^s \alpha_i y_i = x+q$ for some $\alpha_i,q>0$. Represent each $F_i\in \ff_i$ as a walk in the coordinate plane spanned by $y_i$ and $x$. Then we are interested in the property that the ``joint walk (represented by the points $(x,y_1,\ldots, y_s)$) hits the hyperplane $\pi$''.

For example, let us show how the cross $t$-intersecting property is related to ``hitting the plane'' property. Consider two families $\ff_1,\ff_2$ and assume that they hit the plane $\pi$  for $\pi$ defined by $y_1+y_2=x+t$. Then it is easy to see that $\ff_1$ and $\ff_2$ are cross $t$-intersecting: indeed, for any $F_1\in\ff_1,\ F_2\in\ff_2$ find an integer point $\mathbf p:=(x,y_1,y_2)$ such that the joint walk hits the plane $\pi$ at $\mathbf p$. In terms of sets, it means that $|F_1\cap [x]|+|F_2\cap [x]|\ge x+t$, and by pigeon-hole principle, these two sets intersect in at least $t$ elements, even restricted to $[x]$. More importantly, if $\ff_1$ and $\ff_2$ are cross $t$-intersecting {\it and shifted}, then they must hit the plane $\pi$ (cf. \cite{Fra3} for the easy proof)!

Finally, let us describe the idea that is behind the junta approximation result below. For some of the lines, as $x$ grows, it gets increasingly more and more unlikely that a random walk (biased random walk, or random walk with fixed number of diagonal steps) hits this line at point $x$. Thus, a bulk of the family must cross the line for small $x$ and stay under it for large $x$. But it is easy to see that this part of the family is an $x$-junta with the same ``hitting the line'' property. We have a similar situation for several families, i.e., ``sums'' of random walks ``hitting a hyperplane''.

In what follows, all the logarithms have base $e$.

\begin{thm}\label{thmshiftjunta1} Let $n,s\ge 2$ and $k_1,\ldots, k_s$ be positive integers and $q$ be a non-negative real number. Fix some positive reals $\alpha_1,\ldots, \alpha_s$ and a subset of positive integers $S$. For $i\in[s]$, let $\ff_i\subset {[n]\choose k_i}$ be such that for any $F_i\in \ff_i,$ $i\in [s]$, there exists $\ell\in S$ such that
\begin{equation}\label{eqjunta1}\sum_{i=1}^s \alpha_i\big|F_i\cap [\ell]\big|\ge \ell+q.\end{equation}
Fix a positive real $r = r(s)$. Then there exist juntas $\mathcal J_1\subset {[n]\choose k_1},\ldots, \mathcal J_s\subset {[n]\choose k_s}$ with center $[j]$ such that for any $F_i\in\mathcal J_i$, $i\in [s]$  the inequality \eqref{eqjunta1} holds with some $\ell\in S\cap [j]$ and, moreover,
$$|\ff_i\setminus \mathcal J_i|\le \Big(\frac {k_i}n\Big)^r{n\choose k_i},$$
provided one of the three conditions below holds for some  real $\epsilon\in (0,1/3]$.
\begin{enumerate}
  \item[(i)] We have $j:=\max_{i\in[s]}\{\frac {4\sigma}{\epsilon^2k_i}\big(r\log \frac {\sigma}{(1-\epsilon)k_i}+\log\frac 8{\epsilon^2}\big), \frac {\sigma r}{\epsilon k_i}\}$ and $n\ge (1-\epsilon)^{-1}\sigma,$  where $\sigma:=\sum_{i=1}^s\alpha_ik_i$ and $\sigma\ge k_i$ for every $i$.
  \item[(ii)] We have $j:=\max_{i\in[s]}\big\{\frac {4\alpha_is}{\epsilon^2}\big(r\log \frac {\alpha_is}{1-\epsilon}+\log\frac 8{\epsilon^2}\big), r\alpha_is\epsilon^{-1}\big\}$ and $n\ge \max_{i\in[s]}  (1-\epsilon)^{-1}\alpha_isk_i.$
  \item[(iii)] For some $C\ge 2$ we have $j:=\max_{i\in[s]}\alpha_isr\log_C (Ce\alpha_is)$ and $n\ge \max_{i\in[s]}  Ce\alpha_isk_i.$
\end{enumerate}
\end{thm}
{\bf Remarks.}
\begin{itemize}
\item Conditions (i) and (ii) in the theorem coincide when $\alpha_1k_1 = \ldots = \alpha_sk_s$, which is the case in many potential applications.
\item We note that, unlike the previous results, our junta approximations work for essentially the full range of $n$. Indeed, one can easily see that in general these problems only make sense for $n\ge \sum_{i\in [s]} \alpha_ik_i$.

\item  We note that, in general, we can get rid of the $\log$ factor in the definition of $j$ by putting $r=C\log^{-1} (\sigma/k_i)$. However, if we want to get an approximation with, say, $|\ff_i\setminus \mathcal J_i|\le {n-1\choose k_i-1}$, then cannot remove $\log (\sigma/k_i)$ from Theorem~\ref{thmshiftjunta1} and make both $j$ linear in $s$ and $n$ linear in $sk$. Indeed, consider the problem of approximating one $s$-matching-free family and assume that $j=Csr$ and $n=C'sr$ for some absolute $C$, $C'$. ($s$ is sufficiently large.) Clearly, the best center $J$ of size $j$ to take is $[j]$. Then consider the following $s$-matching-free family $\mathcal F:=\{F\in {[n]\choose k}: |F\cap [2j-1]|\ge 2Cr\}$. The part of $\ff$ that is not contained in the junta, has size at least ${j-1\choose 2Cr}{n-2j-1\choose k-2Cr} = \Omega\Big(\big(\frac s2\big)^{2Cr}\big(\frac nk\big)^{2Cr-r}\Big)\cdot {n-r\choose k-r}=\Omega(s^r){n-r\choose k-r}$, where the $\Omega$-notation depends on $C,C',r$ only. Thus, $\ff$ is not well-approximated by the junta, provided $s$ is large enough.
\end{itemize}
For a set $F\subset [n],$ let us define the {\it $p$-biased measure of $F$} by $\mu_p(F) := p^{|F|}(1-p)^{n-|F|}$. For a family $\ff\subset 2^{[n]}$, we put $\mu_p(\ff):=\sum_{F\in\ff}\mu_p(F)$. We can prove the following $p$-biased non-uniform analogue of Theorem~\ref{thmshiftjunta1}. (It should be helpful to think of $p_i$ as $k_i/n$.)

\begin{thm}\label{thmshiftjunta2}
Let $n,s\ge 2$  be positive integers and $q$ be a non-negative real number. Fix some positive real $\alpha_1,\ldots, \alpha_s$ and $p_1,\ldots, p_s\in (0,1)$. Fix a subset of positive integers $S$. For $i\in[s]$, let $\ff_i\subset 2^{[n]}$ be such that for any $F_i\in \ff_i,$ $i\in [s]$, there exists $\ell\in S$ such that $F_i$ satisfy \eqref{eqjunta1}.

Fix a positive real number $r=r(s)$. Then there exist juntas $\mathcal J_1,\ldots, \mathcal J_s\subset 2^{[n]}$ with center $[j]$ such that for any $F_i\in\mathcal J_i$, $i\in [s]$  the inequality \eqref{eqjunta1} holds with some $\ell\in S\cap [j]$ and, moreover,
$$\mu_{p_i}(\ff_i\setminus \mathcal J_i)\le p_i^r,$$
provided one of the three conditions below holds for some  real $\epsilon\in (0,1/3]$.
\begin{enumerate}
  \item[(i)] We have $j:=\max_{i\in[s]}\{\frac {4\sigma}{\epsilon^2p_i}\big(r\log \frac {\sigma}{(1-\epsilon)p_i}+\log\frac 8{\epsilon^2}\big), \frac {\sigma r}{\epsilon p_i}\}$ and $n\ge (1-\epsilon)^{-1}\sigma,$  where $\sigma:=\sum_{i=1}^s\alpha_ip_i$ and $\sigma\ge p_i$ for every $i$.
  \item[(ii)] We have $j:=\max_{i\in[s]}\big\{\frac {4\alpha_is}{\epsilon^2}\big(r\log \frac {\alpha_is}{1-\epsilon}+\log\frac 8{\epsilon^2}\big),r\alpha_is\epsilon^{-1}\big\}$ and $\max_{i\in[s]}  \alpha_isp_i\le (1-\epsilon).$
  \item[(iii)] For some $C\ge 2$ we have $j:=\max_{i\in[s]}\alpha_irs\log_C (Ce\alpha_is)$ and $\max_{i\in[s]}  Ce\alpha_isp_i\le 1.$
\end{enumerate}
\end{thm}
The proof of Theorem~\ref{thmshiftjunta2} is almost the same as that of Theorem~\ref{thmshiftjunta1}. Moreover, the calculations become easier.

\subsection{Combinatorial consequences}
The first consequence is for families satisfying a {\it cross-union} condition. Note that the case $q=1$ corresponds to cross-dependence.
\begin{thm}\label{thmcrossunion} Let $s\ge 2$ be an integer and $\epsilon\in (0,1/3]$, $r=r(s)$ be positive real numbers. For $i\in[s]$, let $\ff_i\subset {[n]\choose k_i}$ be such that $|F_1\cup \ldots \cup F_s|\le k_1+\ldots+ k_s-q$ for some integer $q>0$. Assume that $\ff_i$ are shifted. Then there exist juntas $\mathcal J_1\subset {[n]\choose k_1},\ldots, \mathcal J_s\subset {[n]\choose k_s}$ with center $[j]$ such that for each $i$
$$|\ff_i\setminus \mathcal J_i|\le \Big(\frac{k_i}n\Big)^r{n\choose k_i}$$
and for any $F_i\in\mathcal J_i$, $i\in[s]$, we have $|F_1\cup \ldots \cup F_s|\le k_1+\ldots+k_s-q$, provided one of the following conditions holds.
\begin{enumerate}
  \item[(i)] We have $j:=\max_{i\in[s]}\{\frac {4\sigma}{\epsilon^2k_i}\big(r\log \frac {\sigma}{(1-\epsilon)k_i}+\log\frac 8{\epsilon^2}\big), \frac {\sigma r}{\epsilon k_i}\}$ and $n\ge (1-\epsilon)^{-1}\sigma,$  where $\sigma:=\sum_{i=1}^sk_i$.
  \item[(ii)] We have $j:=\frac {4s}{\epsilon^2}\big(r\log \frac {s}{1-\epsilon}+\log\frac 8{\epsilon^2}\big)$ and $n\ge \max_{i\in[s]}  (1-\epsilon)^{-1}sk_i.$
  \item[(iii)] For some $C\ge 2$ we have $j:=rs\log_C(Ces)$ and $n\ge \max_{i\in[s]}  Cesk_i.$
\end{enumerate}

\end{thm}

Let us derive one particular corollary of this result.
\begin{cor}\label{cor111}
  For each $C, s\ge 2$ and $\epsilon = C^{-\alpha}$, $\alpha>0$ the following holds. If $n=C ks$, and $\ff_1,\ldots, \ff_s \subset {[n]\choose k}$ are shifted cross-dependent families, then there exist cross-dependent $\big(2(1+\alpha)s\big)$-juntas $\mathcal J_1, \ldots, \mathcal J_s$ on the same center, such that
 $|\ff_i\setminus \mathcal J_i|<\epsilon s \binom{n-1}{k-1}$ for each $i$.
\end{cor}
\begin{proof}
  Apply Theorem~\ref{thmcrossunion} (iii) with $k_1 = \ldots = k_s = k$ and $r = \log_{n/k} \frac n{\eps ks} = \log_{Cs} C^{1+\alpha} = \frac{1+\alpha}{\log_C (Cs)}$. Then $|\ff_i\setminus \mathcal J_i|\le \big(\frac kn\big)^r{n\choose k} = \frac{\eps ks}n{n\choose k} = \eps s{n-1\choose k-1}$. At the same time, we have $$j = rs\log_C(Ces) = (1+\alpha)s\frac{\log_C(Ces)}{\log_C(Cs)}< 2(1+\alpha)s,$$
   where the last inequality is valid since $Cs>e$.
\end{proof}

The validity of Corollary~\ref{cor111} for not necessarily shifted families, but with unspecified dependencies between constants $C,\eps$ and $C'$, where the juntas have center of size $C' s$, was 
announced by Keevash, Lifshitz,  Long, and Minzer  as a consequence of general sharp threshold-type results (for the latter, see \cite{KLLM}).

\begin{proof}[Proof of Theorem~\ref{thmcrossunion}] The following proposition is a consequence of the shiftedness of $\ff_i$ (see \cite{Fra3}).
\begin{prop}\label{propsumzero} For any $F_i\in \ff_i$, $i\in [s]$, there exists $\ell$ such that
\begin{equation}\label{eqshift12}
\sum_{i=1}^s |F_i\cap [\ell]|\ge \ell+q.\end{equation} \end{prop}
\begin{proof}[Idea of the proof] Assuming the contrary, it is not difficult to find shifts of $F_i$ that would not satisfy $|F_1\cup\ldots \cup F_s|< k_1+\ldots+k_s-q$.
\end{proof}

Now we can apply Theorem~\ref{thmshiftjunta1} (in one of the three assumptions on $j,n$)  to $\ff_i$ with $\alpha_1=\ldots=\alpha_s=1$ and $S:=[n]$. This gives us the juntas that satisfy \eqref{eqshift12}. However, it is then clear that  for any $F_i\in\mathcal J_i$, $i\in[s]$, we have $|F_1\cup \ldots \cup F_s|\le k_1+\ldots+k_s-q$.
\end{proof}

\begin{thm}
  Let $s\ge 2$, $t\ge 1$ be integers and $\epsilon\in (0,1/3]$, $r=r(s)$ be positive real numbers. For $i\in[s]$, let $\ff_i\subset {[n]\choose k_i}$ be such that $|F_1\cap \ldots \cap F_s|\ge t$. Moreover, assume that $\ff_i$ are shifted. Then there exists juntas $\mathcal J_1\subset {[n]\choose k_1},\ldots, \mathcal J_s\subset {[n]\choose k_s}$ with center $[j]$ such that for each $i$
$$|\ff_i\setminus \mathcal J_i|\le \Big(\frac{k_i}n\Big)^r{n\choose k_i}$$
and such that for any $F_i\in\mathcal J_i$, $i\in[s]$, we have $|F_1\cap \ldots \cap F_s|\ge t$, provided that one of the following conditions holds.

\begin{enumerate}
  \item[(i)] We have $j:=\max_{i\in[s]}\{\frac {4\sigma}{\epsilon^2k_i}\big(r\log \frac {\sigma}{(1-\epsilon)k_i}+\log\frac 8{\epsilon^2}\big), \frac {\sigma r}{\epsilon k_i}\}$ and $n\ge (1-\epsilon)^{-1}\sigma,$  where $\sigma:=\sum_{i=1}^s\frac{k_i}{s-1}$.
  \item[(ii)] We have $j:=\max\big\{\frac {8}{\epsilon^2}\big(r\log \frac s{(s-1)(1-\epsilon)}+\log\frac 8{\epsilon^2}\big), \frac{r s}{\epsilon(s-1)}\big\}$ and $n\ge \max_{i\in[s]}  (1-\epsilon)^{-1}\frac{s}{s-1}k_i.$
  \item[(iii)] We have $j:=3sr$ and $n\ge 2e^2k.$
\end{enumerate}
\end{thm}
\begin{proof}
  The proof is very similar to the previous one. Due to shiftedness of $\ff_i$, for any $F_i\in \ff_i,$ $i\in[r]$, we must have
  $$\sum_{i=1}^s \frac 1{s-1}\big|F_i\cap [\ell]\big|\ge \ell +\frac t{s-1}.$$
  Then apply Theorem~\ref{thmshiftjunta1} with $q:=\frac t{s-1}$, $\alpha=\alpha_i := \frac 1{s-1}$ (and $C = e$ for part (iii)).
\end{proof}

Analogous consequences for non-uniform families can be obtained from the $p$-biased Theorem~\ref{thmshiftjunta2}. Since it is straightforward, we leave the details to the reader.

\subsection{Why do juntas necessarily have center $\mathbf{[j]}$?}\label{sec32}
In this subsection, we prove a simple and intuitive statement that juntas approximating shifted families should have center  $[j]$.

For a given pair of indices $1\le u<v\le n$ and a set $A \subset [n]$, define its \textit{$(u\leftarrow v)$-shift} $S_{u\leftarrow v}(A)$ as follows. If $u\in A$ or $v\notin A$, then $S_{u\leftarrow v}(A) = A$. If $v\in A, u\notin A$, then $S_{u\leftarrow v}(A) := (A-\{v\})\cup \{u\}$. That is, $S_{u\leftarrow v}(A)$ is obtained from $A$  by replacing $v$ with $u$.
The  $(u\leftarrow v)$-shift $S_{u\leftarrow v}(\mathcal A)$ of a family $\mathcal A$ is as follows:
$$S_{u\leftarrow v}(\mathcal A) := \{S_{u\leftarrow v}(A)\, :\,  A\in \mathcal A\}\cup \{A\, :\,  A,S_{u\leftarrow v}(A)\in \mathcal A\}.$$

\begin{prop} Suppose that $\mathcal I\subset 2^{[n]}$ is a junta with center $I$ of size $j$ defined by family $\mathcal I^*$. Assume that $\ff\subset 2^{[n]}$ is shifted and satisfies $|\ff\setminus \mathcal I|\le w$. Then, for any $u< v$, the junta $\mathcal J$ with center $S_{u\leftarrow v}(I)$ and defined by the family $S_{u\leftarrow v}(\mathcal I^*)$ satisfies $|\ff\setminus \mathcal J|\le w$.
\end{prop}
Obviously, after several such $u\leftarrow v$-shifts, the center becomes $[j]$ for $j=|I|$. Moreover, if we are shifting several juntas with center $I$ at the same time, their combinatorial properties are preserved, since the juntas stay the same up to relabeling of the vertices.
\begin{proof}
It is clearly sufficient to deal with the case $I\cap \{u,v\}=v$.  Let us construct an injective map $\phi$ from $\ff\setminus \mathcal J$ into $\ff\setminus \mathcal I$. If $F\in \ff\setminus \mathcal J$ satisfies $|F\cap \{u,v\}|\ne 1$, then $F \in \ff\setminus \mathcal I$ as well, and we put $\phi(F) = F$. If $|F\cap \{u,v\}| = 1$, then $F\in \mathcal I$ if and only if $F\oplus \{u,v\}\in \mathcal J$. We then put $\phi(F) = F\oplus\{u,v\}$ if $F\oplus \{u,v\}\in \ff$. Due to shiftedness, we are only left to deal with $F$, such that $F\cap \{u,v\}=\{u\}$ and such that $F\oplus \{u,v\}\notin\ff$. But if such $F$ does not belong to $\mathcal J$, then $F$ does not belong to $\mathcal I$ (because $S_{u\leftarrow v}(F) = F$). Therefore, we can put $\phi(F) = F$ for such $F$.
\end{proof}
\section{Proofs of Theorems~\ref{thmshiftjunta1}\label{sec4} and~\ref{thmshiftjunta2}}
\subsection{Theorem~\ref{thmshiftjunta1}}
For each $i\in [s]$, let us put
\begin{equation*}\ff'_i:=\big\{F\in \ff_i: \big|F\cap [\ell]\big|\ge \ell\cdot\frac{k_i}{\sigma}\text{ for some }\ell>j\big\},\ \ \ \ \ff''_i:=\ff_i\setminus \ff'_i.\end{equation*}
Let us put $$\mathcal J_i:=\Big\{G\in{[n]\choose k_i}: G\cap [j] = F\cap [j] \text{ for some } F\in \ff''_i\Big\}.$$

The following is the key property of sets in $\ff''_i$.

\begin{prop}\label{propbighelp} For any $F''_1\in\ff''_1\ldots, F''_s\in \ff''_s$, the property \eqref{eqjunta1} holds for some $\ell\in S\cap [j]$.
\end{prop}
\begin{proof}
  If not, then \eqref{eqjunta1} must hold for some $\ell >j$ by the hypothesis. But, by definition of $\ff''_i$, we have $\alpha_i|F''_i\cap [\ell]|<\alpha_ik_i\ell/\sigma$ for any $F''_i\in \ff''_i$ and $\ell>j$, and thus $\sum_{i=1}^s \alpha_i|F''_i\cap [\ell]|<\ell\le \ell+q$, a contradiction.\end{proof}

Let us note the following simple facts:
\begin{itemize}
  \item Each $\mathcal J_i$ is a $j$-junta with center in $[j]$.
  \item We have $\ff''_i\subset \mathcal J_i$.
  \item For any $F_1\in \mathcal J_1,\ldots, F_s\in \mathcal J_s$ we have $\sum_{i=1}^s\alpha_i|F_i\cap [\ell]|\ge \ell+q$ for some $\ell\in S\cap [j]$. Indeed, by Proposition~\ref{propbighelp} this inequality holds for any $F''_1\in \ff''_1,\ldots, F''_s\in \ff''_s$. But, by definition, $F_i\cap [j]=F''_i\cap [j]$ for some $F''_i\in \ff''_i$.
\end{itemize}
To conclude the proof that $\mathcal J_i$ are the desired juntas, we are only left to show that
\begin{prop}\label{prop12}
  We have $|\ff'_i|\le \big(\frac {k_i}n\big)^r{n\choose k_i}$.
\end{prop}

\begin{proof}
Partition $F\in\ff'_i$ into classes by the maximum value of $\ell$, for which $|F\cap [\ell]|\ge k_i\ell/\sigma$. Obviously, $\ell$ must have the form $g(t):=\lfloor \sigma t/k_i\rfloor$ for integer $t\in [ k_ij/\sigma,k_i]$. Since it does not affect the calculations and simplifies the presentation, in what follows we shall treat $\sigma t/k_i$ for each $t$ as an integer (and thus assume that $g(t) = \sigma t/k_i$). Thus, we have
\begin{equation}\label{eqboundfi}|\ff'_i|\le \sum_{t= k_i j/\sigma}^{k_i}{g(t)\choose t}{n-g(t)\choose k_i-t} = \sum_{t = k_ij/\sigma}^{k_i}h\big(t,n,k_i,g(t)\big),\end{equation}
where $h(x,n,k,T) = \Pr[X = x]$ for $X\sim H(n,k,T)$, the hypergeometric distribution $H(n,k,T)$. Recall that $H(n,k,T)$ is the distribution of the size of the intersection of a uniformly random $k$-element subset of $[n]$ with a fixed set of size $T$. Take a random variable $X\sim H\big(n,k_i,g(t)\big)$. Then the expectation $\E X$ of $X$ satisfies $\E X = \frac {g(t)k_i}n = \frac {\sigma t}n\le (1-\eps)t$ due to our choice of $n$. On the one hand, \begin{equation}\label{eqbound3} h(t,n,k_i,g(t))\le \Pr[X\ge t].\end{equation}  On the other hand, it is known (see, e.g., \cite[Theorem 2.10]{Luc}) that \begin{equation}\label{eqmom}\E e^{uX}\le \E e^{uY}\end{equation} for any real $u$ and $Y\sim Bin(g(t),k_i/n)$ (note that $\E Y = \E X\le (1-\eps)t$).

Recall the following Chernoff-type inequality for $Y'\sim Bin(m,p)$ and 
$a>0$ (cf., e.g., \cite[Theorem A.1.11]{AS}):
$$\Pr[Y'-\E Y'\ge a]<e^{-\frac{a^2}{2pm}+\frac {a^3}{2(pm)^2}}.$$
This inequality is proved using the moment generating function, and thus applies to $X$ as well due to \eqref{eqmom}.
Recall that $\epsilon\le 1/3$ and take $Y'\sim Bin(g(t), p)$, where $p\ge \frac{k_i}n$ is such that $\E Y' = (1-\eps) t$. We get that
{\small $$\Pr[X\ge t]\le  \Pr[Y\ge t]\le \Pr\big[Y'-\E Y'\ge \epsilon t\big]<\exp\Big[-\frac {\epsilon^2 t}{2(1-\epsilon)}+ \frac {\epsilon^3 t}{2(1-\epsilon)^2 }\Big]<\exp\Big[-\frac {\epsilon^2 t}{4(1-\epsilon)}\Big]<e^{-\epsilon^2t/4}.$$}

Using the displayed bound above and \eqref{eqbound3}, we get that
$$\sum_{t=k_ij/\sigma}^{k_i} h\big(t,n,k_i,g(t)\big)\le \sum_{t=k_ij/\sigma}^{\infty}e^{-\epsilon^2t/4}\le e^{-\epsilon^2 k_i j/(4\sigma)}\cdot \frac 1{1-e^{-\epsilon^2/4}}\le e^{-\epsilon^2 k_i j/(4\sigma)}\cdot \frac 8{\epsilon^2},$$
where the last bound is due to the fact $e^{-x}\le 1-x/2$ for $x<1$.
Due to our choice of $j$, we get that
$$\sum_{t=k_ij/\sigma}^{k_i} f_t\le e^{-r\log \frac {k_i}{(1-\epsilon)\sigma}-\log\frac 8{\epsilon^2}}\cdot \frac 8{\epsilon^2} = \Big(\frac {(1-\epsilon)k_i}{\sigma}\Big)^r,$$
and therefore $\sum_{t=k_ij/\sigma}^{k_i} f_t\cdot (n/k_i)^r\le 1$ for $n = (1-\epsilon)^{-1} \sigma$. As we have already mentioned, this implies that the same inequality holds for $f_t(n)$ for any $n\ge (1-\epsilon)^{-1} \sigma$. Combining everything together, we have $|\ff'_i|/{n\choose k_i}\le  \sum_{t=k_ij/\sigma}^k f_t(n)\le (k_i/n)^r$. This concludes the proof of the proposition.
\end{proof}
\vskip+0.1cm

The second part of Theorem~\ref{thmshiftjunta1} is proven analogously, with the only difference that we define $\ff_i', \ff_i''$ as follows:
\begin{equation*}\ff'_i:=\big\{F\in \ff_i: |F\cap [\ell]|\ge \ell\cdot \frac{s}{\alpha_i}\text{ for some }\ell>j\big\},\ \ \ \ \ff''_i:=\ff_i\setminus \ff'_i,\end{equation*}
and $g'_i(t):= \alpha_ist$ (again assumed to be integer) plays the role of $g(t)$.
\vskip+0.1cm

For part (iii) of Theorem~\ref{thmshiftjunta1}, we do a slightly different estimate in the proof of Proposition~\ref{prop12}.  Put $t_0:=j/(\alpha_is)$ and note that $t_0\ge r\log_C(Ce\alpha_is)=r+r\log(e\alpha_is)$ by our choice of $j$. We have \begin{equation}\label{eqboundfi2}|\ff'_i|\le \sum_{t= t_0+1}^{k_i}{g'(t)\choose t}{n-g'(t)\choose k_i-t}. \end{equation} Note that if  $g'(t)<t$ then the corresponding term in the sum is zero. Thus, in what follows we assume that $g'(t)\ge t$.
We show that $|\ff_i'|\le \big(\frac {k_i}n\big)^r{n\choose k_i}$. 
Indeed, \begin{equation}\label{eqcompare}\Big(\frac {k_i}n\Big)^r{n\choose k_i}/{n-g'(t)\choose k_i-t}\ge \Big(\frac{n}{k_i}\Big)^{t-r}\ge (C e\alpha_i s)^{t-r}.\end{equation}
   On the other hand,
   $${g'(t)\choose t}\le \Big(\frac{eg'(t)}{t}\Big)^{t}\le (e\alpha_is)^{t}.$$
 Combining the two bounds above, we obtain
 \begin{equation*}\frac{|\ff_i'|}{\big(\frac {k_i}n\big)^r{n\choose k_i}}\ \le  \sum_{t=t_0+1}^{k_i}\frac{(e\alpha_is)^{t}}{(Ce\alpha_i s)^{t-r}}= \sum_{t=t_0+1}^{k_i}\frac{(e\alpha_is)^{r}}{C^{t-r}} \le \frac{(e\alpha_is)^{r}}{C^{t_0-r}} \le \frac{(e\alpha_is)^{r}}{C^{r\log_C(e\alpha_is)}}= 1.\end{equation*}

\subsection{Theorem~\ref{thmshiftjunta2}}
The first part of the proof repeats word for word the proof of Theorem~\ref{thmshiftjunta1}, with an obvious replacement of ${[n]\choose k}$ by $2^{[n]}$ in the definition of $\mathcal J_i$. The only difference is that we need another version of Proposition~\ref{propbighelp}:
\begin{prop}
  We have $\mu_p(\ff'_i)\le p_i^r$.
\end{prop}
\begin{proof}
  With the same definition of $g(t)$ (and the same convention concerning omitting integer parts), we get that

\begin{equation*}\mu_p(\ff'_i)\le \sum_{t= p_ij/\sigma}^{n}{g(t)\choose t}p_i^t(1-p_i)^{g(t)-t}.\end{equation*}
An individual term of this sum is $\Pr[Y= t]$, where $Y\sim Bin(g(t),p_i)$. It is at most $ \Pr\big[Y'-\E Y'\ge \epsilon t\big]$, where $Y'\sim Bin(g(t),p)$ and $p\ge p_i$ is chosen so that $\E Y' = (1-\eps)t$. The remainder of the proof is identical to that of part (i) of Theorem~\ref{thmshiftjunta1}. 

\vskip+0.1cm

The proof goes similarly for the second and the third part of the theorem.
\end{proof}

\section{Families with no cross-matching}\label{sec5}
As an application of Theorem~\ref{thmcrossunion}, we prove an exact bound concerning cross-dependent families. The following question was addressed by Aharoni and Howard \cite{AH}, as well as by Huang, Loh and Sudakov \cite{HLS}. Given $\ff_1\ldots, \ff_s\subset {[n]\choose k}$ that are cross-dependent,\footnote{Recall that cross-dependence means that there are no sets $F_1\in \ff_1,\ldots, F_s\in \ff_s$ that are pairwise disjoint.}
 find $\min_{i\in[s]}|\ff_i|$. (We note here that some authors use the term ``$\ff_1,\ldots, \ff_s$ contain a rainbow matching'' to refer to the situation, opposite to ``cross-dependence''.) In \cite{HLS}, Huang, Loh and Sudakov proved the following result.
\begin{thm}\label{thmhls}
  If $n>3sk^2$ and $\ff_1,\ldots, \ff_s\subset {[n]\choose k}$ are cross-dependent then
  \begin{equation}\label{eqcrossdep}
    \min_{i\in[s]}|\ff_i|\le {n\choose k}-{n-s+1\choose k}.
  \end{equation}
\end{thm}
Later, a version of this theorem for families with different uniformity was  obtained by Lu and Yu \cite{LY}. They have also managed to obtain an inequality of the product of sizes of cross-dependent families, partially answering one of the questions of \cite{HLS}. In \cite{HLW}, Huang, Li and Wang obtained a similar result for a stronger notion of a rainbow matching.

As it is easily guessed from the formula, the bound is attained by the families $\ff_1 = \ldots = \ff_s = \{F\in {[n]\choose k}: F\cap [s-1]\ne \emptyset\}$.

The bound \eqref{eqcrossdep} was obtained for $n>f(s) k$ with some unspecified and very fast growing function $f(s)$ by Keller and Lifshitz in \cite{KLchv} as an application of the junta method. We recommend the reader to consult \cite{KLchv} for more advanced applications of the junta method. The proof of the next theorem follows the framework proposed by the authors of \cite{KLchv}. Our junta approximation gives an almost linear bound.

\begin{thm}\label{thmcrosscross}
The statement of Theorem~\ref{thmhls} holds for $n\ge 12 ks\log (e^2s)$.
\end{thm}

\begin{proof} Since shifting maintains cross-dependence (cf. \cite{Fra3}), we may assume that $\ff_i$, $i\in [s]$, are shifted. Let us apply Theorem~\ref{thmcrossunion} part (iii) with $q=1$, $k_1 = \ldots = k_s = k$, $C=e$ and $r=3$.
We get cross-dependent $j$-juntas $\mathcal J_1,\ldots, \mathcal J_s$ with $j=3s\log (e^2s)$ and such that $|\mathcal J_i\setminus \ff_i|\le \frac {k^3}{n^3}{n\choose k}$. In particular, if $\min_{i\in[s]}|\ff_i|\ge {n\choose k}-{n-s+1\choose k}$, then  $\min_{i\in[s]}|\mathcal J_i|\ge {n\choose k}-{n-s+1\choose k}-\frac {k^3}{n^3}{n\choose k}$.

The proof consists of two steps. The first step is to show that the juntas must have the structure of the (conjectured) extremal family. We will need the following lemma proved by Huang, Loh and Sudakov \cite{HLS}.
    \begin{lem}\label{lemhls} Let $\G_1\subset {X\choose t_1},\ldots, \G_l\subset {X\choose t_l}$ be cross-dependent and $\sum_{i=1}^l t_i\le |X|$. Then there exists $i$ such that $|\G_i|\le (l-1){|X|-1\choose t_i-1}$.
       \end{lem}

 Recall that we use the following notation: for any family $\mathcal G\subset 2^{[m]}$ and $X\subset Y\subset [m]$, put $$\mathcal G(X,Y):=\big\{G\setminus X: G\in \G, G\cap Y = X\big\}.$$

  \begin{lem}
    If $\min_{i\in [s]} |\mathcal J_i|\ge {n\choose k}-{n-s+1\choose k}-\frac {k^3}{n^3}{n\choose k}$, then we have $\mathcal J_1 = \ldots  = \mathcal J_s = \big\{F\in {[n]\choose k}: F\cap [s-1]\ne \emptyset\big\}.$
  \end{lem}

  \begin{proof}
 Consider the defining family $\mathcal J^*_i$ of $\mathcal J_i$.
    If $\mathcal J^*_i$ contains $l$ singletons for some $l\in [s-1]$, then these singletons must be $1,\ldots, l$ due to shiftedness. Clearly, if each $\mathcal J^*_i$ contain $s-1$ singletons, then they cannot contain any sets disjoint to $[s-1]$, and the claim is proved.

    W.l.o.g., assume that $\mathcal J^*_1$ does not contain $\{s-1\}$. Consider a bipartite graph $G$ between $\{1\},\ldots, \{s-1\}$ and $\mathcal J^*_2,\ldots, \mathcal J^*_s$, where an edge between a singleton and a family is drawn if and only if the singleton belongs to the family. Take a maximal (non-extendable) matching in this graph, such that the singletons cover the initial segment $[s-l]$, where $l\ge 1$ is an integer. W.l.o.g., assume the families that are not matched are $\mathcal J^*_1,\ldots, \mathcal J^*_l$ (note that $\mathcal J^*_1$ is not matched since it was not included in the graph). Consider the families
    $\mathcal I^*_i:=\mathcal J^*_i(\emptyset,[s-l])$ and $\mathcal I_i:=\mathcal J_i(\emptyset,[s-l])$, $i=1,\ldots,l$. The former are the families on $2^{X}$ for $X:=[s-l+1,j]$ and none of these families contain a singleton due to the maximality of the chosen matching.

    If $l = 1$ then $\mathcal J^*_1(\emptyset,[s-2])$ can only contain sets containing $s-1$ (otherwise, $\mathcal J_1,\ldots,\mathcal J_s$ are not cross-dependent).  Thus $|\mathcal J_1|\le {n\choose k}-{n-s+1\choose k}-{n-j\choose k-1}<{n\choose k}-{n-s+1\choose k}-\frac {k^3}{n^3}{n\choose k}$, where the last inequality is valid for $n>2kj$.

    From now on, we assume that $l\ge 2$.
    Our next goal is to prove that \begin{equation}\label{eqkeystuff} \text{one of }\mathcal I_1,\ldots,\mathcal I_l\text{ has size smaller than } {n-s+l\choose k}-{n-s+1\choose k}-\frac {k^3}{n^3}{n\choose k}.\end{equation}
     Since $|\mathcal J_i\setminus \mathcal I_i|\le {n\choose k}-{n-s+l\choose k}$, this and $|\mathcal J_i|=|\mathcal J_i\setminus \mathcal I_i|+|\mathcal I_i|$ would lead to a contradiction.

    For each $t=2,\ldots, k$, consider
    $$\mathcal I^{(t)}_i:=\mathcal I_i^*\cap {X\choose t}.$$

    The families $\mathcal I^{(t_1)}_1,\ldots, \mathcal I^{(t_l)}_l$ are cross-dependent for any $t_1\ldots, t_l$.
    The following claim is crucial in bounding the size of $\mathcal I_i$.
    \begin{cla}
      For any $\beta_2,\ldots, \beta_k\ge 0,$ there is $i$ such that $\sum_{t=2}^k\beta_t|\mathcal I_i^{(t)}|\le \sum_{t=2}^k\beta_tl{j-1\choose t-1}$.
    \end{cla}
    \begin{proof}
    Assume that this does not hold. Then, for each $i\in [l]$, there is $t_i$ such that $|\mathcal I_i^{(t_i)}|\ge l{j-1\choose t_i-1}>l{|X|-1\choose t_i-1}$. Since $\mathcal I_i^{(t_i)}\subset {X\choose t_i}$, we have $|\mathcal I_i^{(t_i)}|\le {|X|\choose t_i}$. Together, these two inequalities imply $t_i\le |X|/l$ and thus $\sum_{i=1}^l t_i\le |X|$. Applying Lemma~\ref{lemhls}, we get a contradiction with the fact that $\mathcal I_1^{(t_1)},\ldots,\mathcal I_l^{(t_l)}$ are cross-dependent.
    \end{proof}

Let us put $\beta_t := {n-j\choose k-t}$ for every $t=2,\ldots, k$ and apply the claim. Then there is $i\in [l]$, such that \begin{equation}\label{finalsum}\sum_{t=2}^k{n-j\choose k-t}|\mathcal I^{(t)}_i|\le \sum_{t=2}^k {n-j\choose k-t} l {j-1\choose t-1}.\end{equation} The left hand side of \eqref{finalsum} is exactly $|\mathcal I_i|$! Finally, we are left to estimate the right hand side of \eqref{finalsum}. Denote the $t$-th term of the summation by $f(t)$. Recall that $n\ge 4jk$, and, due to $k\ge 2$ and $j\ge 2$, we have $n-j-k\ge 3jk$. We have
$$\frac{f(t)}{f(t+1)} = \frac{n-j-k+t+1}{k-t}\cdot \frac {t}{j-t}>3t\ge 3.$$
Thus, we get that \eqref{finalsum} and the inequality above implies $|\mathcal I_i|\le \sum_{t=2}^k f(t)< f(1)\cdot\sum_{t=1}^{\infty} (1/3)^t= f(1)\cdot \frac {1/3}{1-1/3} =\frac 12f(1)=\frac l2 {n-j\choose k-1}$. On the other hand, ${n-s+l\choose k}-{n-s+1\choose k}-\frac {k^3}{n^3}{n\choose k}\ge (l-1){n-j\choose k-1}$. Since $l\ge 2$, the last expression is an upper bound on $|\mathcal I_i|$,  which concludes the proof of \eqref{eqkeystuff}.
\end{proof}

The second step is to show that the families $\ff_i$ must satisfy $\mathcal J_i\setminus \ff_i=\emptyset$ for each $i\in[s]$.
This part of the argument is similar to the proofs from \cite{DF, KLchv}.

Assume that, among $i\in [s]$, the density $\alpha_i:=\ff_i(\emptyset,[s-1])/{n-s+1\choose k}$ is the largest for $i=1$ and put $\beta_i^l:=\ff_i(\{l\},[s-1])/{n-s+1\choose k-1}$ for each $i=2,\ldots, s$, $l\in [s-1]$. For a family $\mathcal G\subset{[m]\choose k}$ and $t\ge k$, let $\bar\partial^t\mathcal G$ be the collection of all sets in ${[m]\choose t}$ that contain at least one set from $\mathcal G$.  The following analytic corollary of the Kruskal--Katona \cite{Kr,Ka} theorem was proved by Bollob\'as and Thomason \cite{BT}:
$$\Big|\bar\partial^t\mathcal G/{n\choose t}\Big|^{n-k}\ge \Big|\mathcal G/{n\choose k}\Big|^{n-t}.$$
From here, we conclude that $\alpha'_1:=|\bar\partial^{(n+k)/2} \ff_1(\emptyset, [s-1])|\ge \alpha_1^{1/2}$.

At the same time, for any bijection $\pi:[s-1]\to [2,s]$, the families $\ff_{\pi(i)}(\{i\},[s-1]),$ $i\in[s-1]$, and $\bar\partial^{(n+k)/2}\ff_1(\emptyset,[s-1])$ are cross-dependent. Since $(n+k)/2+(s-1)k<n$, we may take a random ordered matching consisting of $s-1$ sets  $M_1,\ldots, M_{s-1}$ of size $k$ and one set $M_s$ of size $(n+k)/2$. In any such matching, there are at most $s-1$ indices $i\in [s]$ such that:  $M_i\in \ff_{\pi(i)}(\{i\},[s-1])$  for $i\in[s-1]$, or   $M_i\in \bar\partial^{(n+k)/2}\ff_1(\emptyset,[s-1])$ for $i=s$. Computing the expectation of the number of such indices, we get
$$\alpha'_1+\sum_{i=1}^{s-1}\beta_{\pi(i)}^i\le s-1.$$
Therefore, there exists $i\in [2,s]$, such that $\sum_{l=1}^{s-1}\beta_i^l\le s-1-\alpha'_1$. This implies that
$$|\ff_i|\le {n\choose k}-{n-s+1\choose k}+\frac {k^3}{n^3}{n\choose k}-\frac {k^{3/2}}{n^{3/2}}{n-s+1\choose k-1}<{n\choose k}-{n-s+1\choose k},$$
a contradiction. Here the last inequality is again due to our choice of $n$, that is, $n>12ks\log(e^2s)$.
\end{proof}

Thanks to the Kruskal--Katona theorem, we can derive the following non-uniform generalization of Theorem~\ref{thmcrosscross}.

\begin{thm}
  Fix integers $k_1,\ldots, k_s$ and assume that $n\ge 12ks\log(e^2s)$ for $k:=\max_{i\in[s]}k_i$. Assume that $\ff_1\subset {[n]\choose k_1},\ldots, \ff_s\subset {[n]\choose k_s}$ are cross-dependent. Then there exists $i\in [s]$ such that
  $$|\ff_i|\le {n\choose k_i}-{n-s+1\choose k_i}.$$
\end{thm}
\begin{proof}
  Suppose that $|\ff_i|>{n\choose k_i}-{n-s+1\choose k_i}$ for every $i$. Then the families $\bar \partial^k\ff_1,\ldots, \bar \partial^k\ff_s$ are cross-dependent as well. Moreover, we claim that $$|\bar \partial^k\ff_i|>{n\choose k}-{n-s+1\choose k}$$ for every $i\in[s]$. This would contradict Theorem~\ref{thmcrosscross}, so we only need to verify the displayed inequality.

  Set $\G_i:=\{[n]\setminus F:F\in \ff_i\}$.
  $$|\G_i|>{n-1\choose k_i-1}+\ldots+{n-s+1\choose k_i-1} = {n-1\choose n-k_i}+\ldots+{n-s+1\choose n-k_i-s+1}.$$
  By the Kruskal--Katona theorem \cite{Kr,Ka}, its shadow on level $n-k$ will be
  $$|\partial ^{n-k}\G_i|>{n-1\choose n-k}+\ldots+{n-s+1\choose n-k-s+1} = {n\choose k}-{n-s+1\choose k}.$$
  But $|\partial ^{n-k}\G_i| = |\bar \partial^k\ff_i|.$
\end{proof}

\section{Concluding Remarks}
We believe that Theorem~\ref{thmcrosscross} is only one example of the many possible scenarios where the shifted juntas enable one to get better, concrete bounds which seem to be elusive for the general juntas obtained by means of discrete Fourier transform methods.

We note that the validity of Theorem~\ref{thmcrosscross} for $n>Csk$ with some large $C$ was announced by Keevash, Lifshitz,  Long, and Minzer  as a consequence of general sharp threshold-type results. In the proof, they use the aforementioned variant of Corollary~\ref{cor111}.\vskip+0.1cm

\section*{Acknowledgments} We thank Noam Lifshitz and the anonymous referee for many helpful remarks.

The authors' research was supported by the grant of the Russian Government N 075-15-2019-1926. The research of the second author was in part supported by the Advanced Postdoc.Mobility grant no. P300P2\_177839 of the Swiss National Science Foundation.

\bibliographystyle{amsplain}


\begin{dajauthors}
\begin{authorinfo}[peter]
  Peter Frankl\\
  R\'enyi Institute\\
  Budapest, Hungary\\
  peter.frankl\imageat{}gmail\imagedot{}com \\
  \url{https://users.renyi.hu/~pfrankl/}
\end{authorinfo}
\begin{authorinfo}[andrey]
  Andrey Kupavskii\\
  IAS, Princeton,\\
  CNRS, France,\\
  MIPT, Moscow
  kupavskii\imageat{}ya\imagedot{}ru \\
  \url{http://kupavskii.com}
\end{authorinfo}
\end{dajauthors}

\end{document}